\documentclass[10pt]{amsart}
\usepackage{array, amsmath, enumerate, color}
\usepackage{amssymb}
\usepackage{mathrsfs}
\usepackage{graphicx,subfigure}

\makeatletter
 \usepackage{amsthm}

\usepackage{fullpage}

\makeatother

\newtheorem{case}{Case}
\newtheorem{subcase}{Subcase}[case]

\newtheorem{thm}{Theorem}[section]
\newtheorem{Lemma}[thm]{Lemma}
\newtheorem{lemma}[thm]{Lemma}

\newtheorem{conj}[thm]{Conjecture}

\newcommand{\sU}{\mathcal{U}}
\newcommand{\sX}{\mathcal{X}}
\newcommand{\sY}{\mathcal{Y}}
\newcommand{\sZ}{\mathcal{Z}}
\newcommand{\sS}{\mathcal{S}}
\newcommand{\sT}{\mathcal{T}}
\newcommand{\sA}{\mathcal{A}}

\title{Strong chromatic index of graphs with maximum degree four}

\author[Huang]{Mingfang Huang}
\thanks{The first author's research is supported by the Fundamental Research Funds for the Central Universities(WUT: 2018IA003,2017IB014).}
\address{Department of Mathematics, School of Science,
Wuhan University of Technology,
Wuhan, China 430070}

\author[Santana]{Michael Santana}
\address{Department of Mathematics,
Grand Valley State University,
Allendale, Michigan 49401}

\author[YU]{Gexin Yu}
\thanks{The third author's research is supported in part by the NSA grant  H98230-16-1-0316 and NSFC grant (11728102).}
\address{Department of Mathematics,
The College of William and Mary,
Williamsburg, VA 23185}
\email{gyu@wm.edu}

\date{\today}

\begin{document}

\begin{abstract}
A strong edge-coloring of a graph $G$ is a coloring of the edges such that every color class induces a matching in $G$.  The strong chromatic index of a graph is the minimum number of colors needed in a strong edge-coloring of the graph.  In 1985, Erd\H{o}s and Ne\v{s}et\v{r}il conjectured that every graph with maximum degree $\Delta$ has a strong edge-coloring using at most $\frac{5}{4}\Delta^2$ colors if $\Delta$ is even, and at most $\frac{5}{4}\Delta^2 - \frac{1}{2}\Delta + \frac{1}{4}$ if $\Delta$ is odd.  Despite recent progress for large $\Delta$ by using an iterative probabilistic argument, the only nontrivial case of the conjecture that has been verified is when $\Delta = 3$, leaving the need for new approaches to verify the conjecture for any $\Delta\ge 4$. In this paper, we apply some ideas used in previous results to an upper bound of 21 for graphs with maximum degree 4, which improves a previous bound due to Cranston in 2006 and moves closer to the conjectured upper bound of 20.

\end{abstract}
\maketitle

\section{ introduction}

All graphs considered in this paper are finite, loopless, undirected, and may have multiple edges.  For a graph $G$, we use $V(G)$ and $E(G)$ to denote the set of vertices and edges of $G$, respectively, and we use $\Delta(G)$ to denote the maximum degree of $G$.  First introduced by Fouquet and Jolivet \cite{[FJ]}, a \emph{strong edge-coloring} of a graph $G$ is an assignment of colors to the edges of $G$ such that if edges $e_1$ and $e_2$ receive the same color, they cannot be incident with one another nor can they be incident with a common edge.  Thus, every color class in a strong edge-coloring induces a matching in $G$. The \emph{strong chromatic index} of a graph $G$, denoted by $\chi'_s(G)$, is the minimum number of colors necessary for a strong edge-coloring of $G$.  Observe that the strong chromatic index of $G$ is equivalent to the chromatic number of $L^2(G)$, which is the square of the line graph of $G$.

Via the greedy algorithm, we see that  $\chi_s'(G) \leq 2 \Delta ^2 - 2 \Delta +1$ for every graph $G$ with maximum degree $\Delta$. In 1985, Erd\H{o}s and Ne\v{s}et\v{r}il \cite{[EN]} conjectured the following upper bounds:

\begin{conj} {\rm (Erd\H{o}s and Ne\v{s}et\v{r}il \cite{[EN]}) } \label{conj:any-g} For every graph $G$ with maximum degree $\Delta$,

\begin{equation*}
 \chi_s'(G) \leq
\begin{cases}
 \frac{5}{4} \Delta^2 & \text{if } \Delta \text{ is even},\\
\frac{5}{4}\Delta^2 - \frac{1}{2}\Delta + \frac{1}{4} & \text{if } \Delta \text{ is odd}.
\end{cases}
\end{equation*}
\end{conj}

Erd\H{o}s and Ne\v{s}et\v{r}il showed further that this conjecture, if true, is best possible by constructing a particular blow-up of $C_5$.  It is worth noting that if a graph $G$ is $2K_2$-free, then $\chi'_s(G) = |E(G)|$.  In 1990, Chung, Gy\'{a}rf\'{a}s, Trotter, and Tuza \cite{CGTT} showed that the maximum number of edges in a $2K_2$-free graph with maximum degree $\Delta$ is $\frac{5}{4}\Delta^2$ for even $\Delta$, and $\frac{5}{4}\Delta^2 - \frac{1}{2} + \frac{1}{4}$ for odd $\Delta$; furthermore, the aforementioned blow-up of $C_5$ is the unique graph that attains this maximum.

While Conjecture \ref{conj:any-g} has been the impetus for many other conjectures and results in the area of strong edge-colorings (see \cite{BI, BQ, FGST, HMRV, HV, HYZ, KLRSWY, SY, WL, Yu}  for only a few), not much progress has been made in regards to proving this conjecture directly.    The first nontrivial case of Conjecture \ref{conj:any-g} (i.e., for graphs with maximum degree at most three) was verified by Andersen \cite{[Andersen]} and independently by Hor\'{a}k, Qing, and Trotter \cite{[HQT]}.  For graphs with maximum degree at most four, Hor\'{a}k \cite{[Ho]} first proved an upper bound of 23 in 1990.  This was later improved by Cranston \cite{[Cr]} in 2006, who showed that 22 colors suffice, which is $2$ away from the conjectured bound $20$.

For graphs with large enough $\Delta$, exciting progress has been made.  In 1997, Molloy and Reed \cite{MR} showed that such a graph $G$ has $\chi'_s(G) \le 1.998\Delta^2$.  In 2015, Bruhn and Joos \cite{BJ} improved this bound to $1.93\Delta^2$. Very recently, Bonamy, Perrett, and Postle \cite{BPP} improved it to $1.835\Delta^2$.    All of these proofs considered the coloring of $L^2(G)$, in which each vertex has a sparse neighborhood (with at most $0.75{2\Delta^2\choose 2}$ edges), and then used an iterative coloring procedure.  However, as pointed out in \cite{MR},  this method cannot get a bound better than $1.75\Delta^2$.   Therefore, it is necessary to explore new approaches and ideas to attack the conjecture.

We turn to the first unsolved case, $\Delta=4$. We develop some ideas hidden in \cite{[Andersen]} by Andersen and prove the following.

\begin{thm}
\label{thm:main}
For every graph $G$ with maximum degree four, $\chi_s'(G)\leq21$.
\end{thm}

The idea of the proof is as follows.  For a minimum counterexample $G$, we construct a partition $V(G) = L \cup M \cup R$ such that:

(1) For any $u\in L$ and $v\in R$ ,   the distance between $u$ and $v$ is at least two, and

(2) the vertices in $M$ are all within distance two from a fixed vertex.

By (1), we can color the edges in $G[L]$ and $G[R]$ independently, but also `collaboratively', and by (2), a coloring on $G[L]$ and $G[R]$ can be extended to the whole graph, because the edges incident with $M$ have clear structures.   We hope this idea can stimulate new ideas to attack Conjecture~\ref{conj:any-g}.

The paper is organized as follows.  
In Section \ref{sec:decom} we introduce some notation and prove various strutural statements about a minimal counterexample $G$.  In particular, we show that the girth of $G$ is at least six, whose proof is in Section~\ref{proof-le2}.  In Section \ref{sec:partition}, we obtain the partition described above.  In Section~\ref{sec:proof},  we show how to color the edges in $G[L]$ and $G[R]$ `collaboratively', and extend it to a coloring of the whole graph; this completes the proof of Theorem~\ref{thm:main}.

\section{Notation and some properties of minimal counterexamples} \label{sec:decom}

We will use the following notation.    For  two disjoint subsets of $V(G)$, call them $X$ and $Y$,  we let $E(X,Y)$ denote the set of edges of $G$ with one end in $X$ and the other end in $Y$.  For an edge $e=uv$, we let $N_1(e)$ be the set of edges incident with $u$ or $v$ in $G-{e}$, and we let $N_2(e)$ be the set of edges not in $N_1(e)$ that have an endpoint adjacent to either $u$ or $v$  in $G-{e}$. We denote the set of edges of $N_1(e) \cup N_2(e)$ by $N(e)$, so that $N(e)$ contains at most $24$ edges in a graph with maximum degree at most four.  Furthermore, if $e' \in N(e)$, we will say that $e$ \emph{sees} $e'$ and vice-versa.

A \emph{partial strong edge-coloring} (or we will sometimes say a \emph{good partial coloring}) of $G$ is a coloring of any subset of $E(G)$ such that if any two colored edges $e_1$ and $e_2$ see one another in $G$, then $e_1$ and $e_2$ receive different colors.  In particular, if a partial strong edge-coloring spans all of $E(G)$, then it is a strong edge-coloring of $G$.  Given a partial strong edge-coloring of $G$, call it $\phi$, we define $A_\phi(e)$ to be the set of colors available for edge $e$.

In the rest of this paper, we assume that G is a minimal counterexample with $|V(G)|+|E(G)|$ minimized.  Here are some structural lemmas regarding $G$.

\begin{lemma}\label{le1}
$G$ is 4-regular.
\end{lemma}

\begin{proof}
Suppose on the contrary that $v$ is a vertex of degree at most three with $N(v) \subseteq \{u_1,u_2,u_3\}$.  By the minimality of $G$, $G - v$ has a good coloring.  Observe that $|A(u_iv)| \ge 3$ for $i \in [3]$.  Thus, we can color the remaining edges in any order to obtain a good coloring of $G$. This is a contradiction.
\end{proof}

\begin{Lemma}\label{le21}
$G$ contains no edge cut with at most $3$ edges.
\end{Lemma}

\begin{proof}
 Suppose otherwise that $G$ contains a smallest edge cut with at most $t\le 3$ edges, say $e_1=a_1b_1, \ldots, e_t=a_tb_t$. By the minimality of $G$, $G$ is connected.  So $G-\{e_1, \ldots, e_t\}$ contains two components, say $G_1$ and $G_2$, so that $a_1, \ldots, a_t\in G_1$ and $b_1, \ldots, b_t\in G_2$.  Note that $a_t$'s and $b_t$'s may be not distinct.   Let $G_1'$ be the graph obtained from $G_1$ by adding vertex $z_1$ and edges $z_1a_1, \ldots, z_1a_t$.  Similarly, let $G_2'$ be the graph obtained from $G_2$ by adding vertex $z_2$ and edges $z_2b_1, \ldots, z_2b_t$.  By the minimality of $G$, both $G_1'$ and $G_2'$ can be colored with $21$ colors.

By renaming the colors, we may assume that $z_1a_s$ and $z_2b_s$ have the color $s$ for each $1\le s\le t\le 3$.  Again by renaming colors, we may assume that the colors appearing on edges incident with $a_1, a_2, \ldots, a_t, b_1, \ldots, b_t$ are all different, which is possible, since there are at most $18$ such edges but there are $21-t\ge 18$ colors other than $1,\ldots,t$. Now, we can obtain a coloring of $G$ by combining the colorings of $G_1'$ and $G_2'$:  keep the colors of the edges in $G_1$ and $G_2$, and color $e_1, \ldots, e_t$ with $1,\ldots, t$, respectively.  This is a contradiction.
\end{proof}

The girth of a graph $G$  is the length of its shortest cycle.

\begin{Lemma}\label{le2}
The graph $G$ has girth at least six.
\end{Lemma}

Since the proof of this lemma is long, we devote Section~\ref{proof-le2} to it.  The reader may skip the proof for now.

By Lemma~\ref{le2}, we may assume that $G$ is a simple graph.

\section{A partition of the vertices}\label{sec:partition}

Let $x$ be any vertex of $G$.  In this section, we consider a coloring strategy that leads to a partition of $V(G)$ into sets $L, M$, and $R$, such that there are no edges between $L$ and $R$,  the numbers of the edges in $E(L,M)$ and $E(M,R)$ are relatively small, and $M$ only contains some vertices within distance $2$ from $x$.   By Lemma~\ref{le1}, $G$ is 4-regular. So we let $N(x)=\{u,v,w,y\}$ and for $z\in N(x)$,  $N(z)=\{z_1,z_2,z_3,x\}$. By Lemma~\ref{le2}, above all these vertices are distinct. Furthermore, we let  $N(z_i)=\{z_{i1},z_{i2},z_{i3},z\}$  for $z\in N(x)$ (see Figure 1). Note that for $i,j,k,\ell \in \{1,2,3\}$ and $a,b \in \{u,v,w,y\}$, $a_{ij}, b_{k\ell}$ may be identical when $a\neq b$.

\medskip

We now give a partial strong edge-coloring of $G$, call it $\psi$, using three colors: assign the edges $uu_1,\ vv_1, \ ww_1$  with the color $1$, assign the edges $uu_2, \ vv_2$  with the color $2$, and assign the edges $uu_3, \ vv_3$  with the color $3$.

Consider the sequence $S_0$ of edges: $w_2w_{21},\ w_3w_{31}, \ ww_2,\ ww_3,\ xu,\ xv,\ xy,\ xw$. We extend $S_0$ to a sequence $S$ of uncolored edges such that the following hold:
\begin{enumerate}[(i)]
\item $S$ contains $S_0$, where $S_0$ is at the end of $S$;
\item for each edge $e$ of $S-S_0$, at least $4$ edges of $N(e)$ fall behind it in $S$;
\item among all sequences satisfying (i) and (ii), $S$ is longest.
\end{enumerate}

Observe that no edge outside of $S$ can see four edges in $S$, otherwise it could be added to the start of $S$ and contradict (iii).

\begin{figure}[h]
\centering
\includegraphics[scale=0.62]{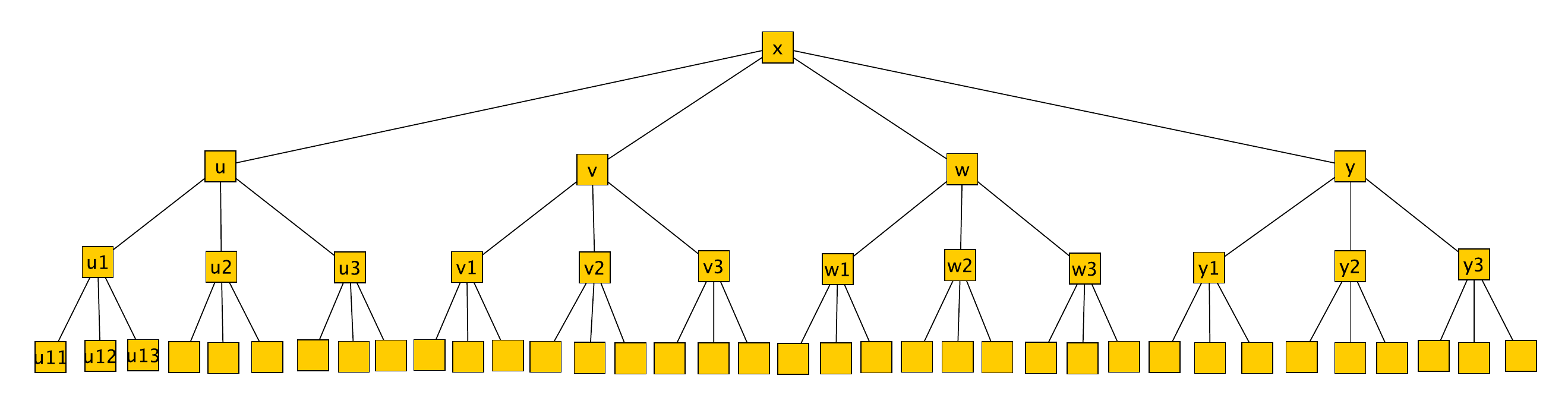}
\caption{{\footnotesize  4-regular graph}}\label{4-regular}
 \end{figure}

\begin{Lemma}
\label{le3}
 With $21$ colors, we may extend $\psi$ to a partial strong edge-coloring of $G$ that inlcudes all edges of $S$.
\end{Lemma}

\begin{proof}
Using 21 colors, greedily color the edges of $S$ in order, and let $e$ be the first edge of $S$ that cannot be colored.  Let $\phi$ denote this partial strong edge-coloring of $G$.  Observe that $e$ must be in $\{w_2w_{21}, w_3w_{31}, xu, xv, xy, xw\}$, as otherwise $|A_\phi(e)| \ge 21 - (|N(e)| - 4) = 1$, so that $e$ can be colored.  Further, by the repetition of colors on the pre-colored edges, $e \notin \{xu, xv, xy, xw\}$.  Thus, it suffices to consider $e \in \{w_2w_{21}, w_3w_{31}\}$.

Without loss of generality, assume that $e = w_2w_{21}$.  Since $e$ cannot be colored, it follows that the 21 colored edges in $N(w_2w_{21})$ must be assigned 21 different colors. Thus, we can remove the color 1 from $ww_1$ and assign it to $w_2w_{21}$.  Observe that in this new partial strong edge-coloring, $w_3w_{31}$ sees at least 4 uncolored edges, and $ww_1$ sees at least 6 uncolored edges.  Hence, we can color $w_3w_{31}$ and recolor $ww_1$. Since $xw$ sees $w_2w_{21}$ colored with 1, there is a color available for the remainder of $S$ by the repetition of colors on the pre-colored edges.
\end{proof}

%

By Lemma~\ref{le3}, if $S$ contains all uncolored edges of $G$ under $\psi$, then we are done. So we assume that $S$ does not contain all uncolored edges of $G$.  Let $H$ be the set of uncolored edges not in $S$, and let $L$ be the set of endpoints of the edges in $H$. Then $L\neq \emptyset$. By the maximality of $S$, $w_2w_{22}$ appears in $S$ since $w_2w_{21}, ww_2, ww_3$ and $xw$ are in $S_0$. Similarly,  $w_2w_{23},w_3w_{32},w_3w_{33},yy_1,yy_2,yy_3$ appear in some order in $S$. So,  all edges incident with $x,u,v,w,y,w_2,w_3$ are either pre-colored or in $S$. By the definition of $L$, $x,u,v,w,y,w_2,w_3\notin L$.

\begin{Lemma}
\label{le4}
$E(G[L])=H$.
\end{Lemma}
\begin{proof}
Suppose otherwise that there exists an edge $e\in E(G[L])$ with endpoints $a$ and $b$ such that $a,b\in L$ but  $e\notin H$.  Let $N(a)=\{a_1,a_2,a_3,b\}$ and $N(b)=\{b_1,b_2,b_3,a\}$ where $aa_1,bb_1\in H$. Since $x,u,v,w,y,w_2,w_3\notin L$, every pre-colored edge and every edge of $S_0$ cannot join two vertices of $L$. So, $e\in S-S_0$. By the definition of $S$, at least 4 edges, say $e_1,e_2,e_3$ and $e_4$, of $N(e)$ are in $S$.  If say $e_1$ belongs to $N_1(e)$, then either $aa_1$ or $bb_1$ sees $e, e_1$, and two edges from $\{e_2, e_3, e_4\}$.  That is, either $aa_1$ or $bb_1$ can be added to $S$, which contradicts the maximality of $S$.

Therefore, $e_1,e_2,e_3,e_4\in N_2(e)$.
Furthermore, we claim that exactly two of these edges are incident with vertices in $\{a_1, a_2, a_3\}$, otherwise either $aa_1$ or $bb_1$ sees three of these edges along with $e$, and so is in $S$.  Without loss of generality, assume that $e_1, e_2$ are incident with vertices in $\{a_1, a_2, a_3\}$. Let's further assume that $e_1$ is behind $e_2$ in the sequence $S$, and let $e_1 = a_ia_{i1}$ for some $i \in [3]$.  Observe that $aa_2, aa_3 \notin S$, as otherwise $aa_1$ would see four edges in $S$, and so be in $S$.

We now show that $e_1$ is not in $S_0$, as otherwise  one of the endpoints of $e_1$ is incident with four edges in $S$.  Thus, $aa_i$ would see each of these four edges and so be in $S$, which is a contradiction.

By the definition of $S$, at least three edges of $N(e_1)$ different from $e$ are behind $e_1$ in $S$.  We next assume that there is at least one edge of these edges incident with $a_{i1}, a_{i2}, a_{i3}$. Since $e_1,e_2$ and $e$ are in $S$, $aa_i\in S$, a contradiction.  So, all these three edges are incident with $N(a_{i1})\backslash\{a_i\}$. However, all four edges incident with $a_{i1}$ would see these three edges together with $e_1$, so that four edges incident with $a_{i1}$ are in $S$. Thus,  $aa_i \in S$, again a contradiction.
\end{proof}

 Let $F=E(L,G-L)$ and $A=\{u_1,u_2,u_3,v_1,v_2,v_3,w_1\}$. We present the relationship between edges of $F$ and vertices of $A$ as follows.

\begin{Lemma}
\label{le5}
 Each edge of $F$ is incident with exactly one vertex of $A$, and each vertex in  $A$ is incident with at most two edges of $F$. Moreover, no vertex in $L$ is incident with two edges of $F$.
\end{Lemma}
\begin{proof}
First note that if $e$ is an edge in $F$ with endpoints $z \in L$ and $z' \in V(G - L)$, then $z'$ must be incident with a pre-colored edge by $\psi$.  If not, then every edge incident with $z'$ is in $S$, and consequently, every edge incident with $z$ is in $S$, by the maximality of $S$.  Yet this contradicts $z \in L$.

Now suppose $e$ is an edge of $F$. Then $e$ is incident with at most one vertex of $A$. Otherwise, the girth of $G$ is at most 5, contrary to Lemma~\ref{le2}. Now we show that $e$ is incident with at least one vertex of $A$.  As shown above, one of the endpoints of $e$ must be incident with a pre-colored edge.  We are done unless $e \in \{xu, xv, xw\}$.  Yet $x,u, v, w \in V(G - L)$, which contradict that $e$ is an edge of $F$.  Therefore, each edge of $F$ is incident with exactly one vertex of $A$.

Next we show that each vertex in  $A$ is incident with at most two edges of $F$.  Suppose otherwise that a vertex $a\in A$ is incident with three edges of $F$. Assume that $a\in L$. Since $u,v,w\in V(G-L)$, one edge of these three edges is pre-colored and other two edges are uncolored. Let $aa'$ be such an uncolored edge where $a'\in V(G-L)$.  By Lemma \ref{le2}, $a'$ is not incident with a pre-colored edge.  Thus, every edge incident with $a'$ is in $S$.  Yet this would imply that every uncolored edge incident with $a$ is also in $S$, contrary to the assumption that $a \in L$.

So we assume that $a\in V(G-L)$. Since $u,v,w\in V(G-L)$, three edges, say $e_1=aa_1,e_2=aa_2$ and $e_3=aa_3$ where $a_1,a_2,a_3\in L$, are the three uncolored edges of $F$ incident with $a$.   Since $a \in V(G - L)$, $e_1, e_2, e_3 \in S$, and further, $e_1,e_2,e_3\notin S_0$.  We assume, without loss of generality, that both $e_1$ and $e_2$ preceed $e_3$ in the sequence $S$. By the definition of $S$, at least $4$ edges of $N(e_3)$ come after $e_3$ in $S$. However, one of these four edges together with $e_1, e_2, e_3$,  are seen by all four edges incident with either $a_1, a_2$, or $a_3$.  Thus, at least one of $a_1, a_2, a_3$ is incident with four edges in $S$, which contradicts $a_1,a_2,a_3\in L$.

We finally show that no vertex in $L$ is incident with two edges of $F$. Suppose otherwise that the vertex $z\in L$ is incident with
two edges of F, and let $z'$ and $z''$ be the other endpoints of these edges.  As shown at the start of this proof, $z'$ and $z''$ are incident with pre-colored edges.  As a consequence, $z \notin A$, and further $z \neq x$.  Thus, $z', z'' \in A \cap V(G - L)$.  So $z'$ and $z''$ are surrounded by 3 edges in $S$, respectively.  Yet every edge incident with $z$ sees these edges in $S$, and so $z \notin L$, a contradiction.
\end{proof}

Let $V_F, V_F'$ be the endpoints of $F$ in $G-L$ and in $L$, respectively. So $F=E(G-L, L)=E(V_F, V_F')$.  Let $M=\{x,u,v,w\}\cup V_F$ and $R=V(G)-L-M$. (See Figure~\ref{figure2} for an example.)  Observe that $E(L,R)=\emptyset, E(L,M)=F$ and $y,w_2,w_3\in R$.  Furthermore, if $z \in V_F$, then $z$ is incident with a pre-colored edge under $\psi$; for otherwise, $z$ is incident with four edges in $S$ and its neighbor in $L$ would then be incident with four edges in $S$, a contradiction.  Thus, $V_F \subseteq A \cup \{x,u,v,w\}$.

\begin{figure}[h]
\includegraphics[scale=0.3]{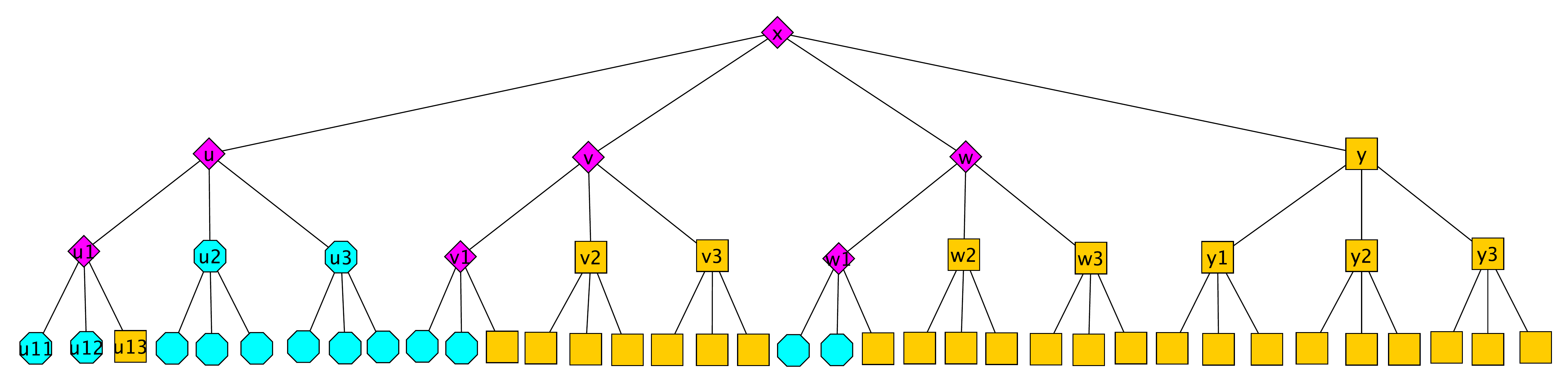}
\caption{A possible partition of the vertices with  $M=\{x, u, v, w, u_1, v_1, w_1\}$ (diamond vertice),  $R$ (square vertices) and $L$ (octagon vertices), where $F=\{uu_2, uu_3, u_1u_{11}, u_1u_{12}, v_1v_{11}, v_1v_{12}, w_1w_{11}, w_1w_{12}\}$ and $V_F=\{u, u_1, v_1, w_1\}$. }
\label{figure2}
\end{figure}

An important observation is that no edges from $G[L]$ and $G[R]$ see each other, so they can be colored independently and be combined together without the need of changing their colors. Now we state some straightforward results as follows.

 \begin{Lemma}
\label{le6} For $z\in \{u,v,w\}$ and $i,j,k\in [3]$, each of the following holds.
\begin{enumerate}[(1)]
\item If $z_i\in M$, then for some $k\neq j$, $z_iz_{ij}\in F$ and $z_iz_{ik}\in E(M,R)$.
\item If $z_iz_{ij}\in F$, then $z_i\in M,z_{ij}\in L$ and three edges incident with $z_{ij}$ are in $G[L]$.
\item If $z_i\in L$, then $z_iz_{ij}\in E(G[L])$.
\item If $z_i\in R$, then $z_iz_{ij}\in E(G[R])$. Further, $yy_j\in E(G[R])$.
\item If $z_iz_{ij}\in E(M,R)$, then $z_i\in M, z_{ij}\in R$ and at least one edge incident with $z_{ij}$ is in $G[R]$.
\item If $z_i\neq w_1$ and $z_iz_{ij},z_iz_{ik}\in E(M,R)$, then at least three of the eight edges incident with $z_{ij}$ and $z_{ik}$ are in $G[R]$.
\end{enumerate}
\end{Lemma}

\begin{proof}
Observe that if $z_i \in M$, then $z_i \in V_F$ and consequently,  $z_i \in A$ by Lemmas \ref{le2} and \ref{le5}.

Lemma \ref{le5} implies (1) as every vertex in $A$ is incident with at most two edges of $F$.

If $z_iz_{ij} \in F$, then $z_i \notin L$, else $z_i$ would be incident with two edges of $F$, namely $z_iz_{ij}$ and $z_iz$, contradicting Lemma \ref{le5}.  Therefore, $z_i \in M$ and $z_{ij} \in L$.  Further, every edge incident with $z_{ij}$ other than $z_iz_{ij}$ must be in $E(G[L])$.  This proves (2).

If $z_i \in L$, then $z_i \in A$ since $w_2, w_3 \in V(G - L)$.  Thus, $zz_i \in F$, and every other edge incident with $z_i$ must be in $E(G[L])$ by Lemma \ref{le5}.  This proves (3)

If $z_i \in R$ and $z_{ij} \notin R$, then $z_{ij} \in M$.  In particular, $z_{ij} \in V_F$ so that $z_{ij}$ is incident with a pre-colored edge under $\psi$. Yet this contradicts Lemma \ref{le2}. Thus,  $z_iz_{ij}\in E(G[R])$.
 Further, notice that $y\in R$. If
 $y_j \notin R$, then $y_j \in M$. So $y_j \in V_F$.  By Lemma \ref{le5}, $y_j$ is incident with a vertex of A. This contradicts Lemma \ref{le2}. Thus, $yy_j\in E(G[R])$.  This proves (4).

If $z_iz_{ij} \in E(M,R)$ and $z_{ij} \in M$, then $z_{ij} \in V_F$ and is incident with a pre-colored edge under $\psi$.  This contradicts Lemma \ref{le2} as previously.  So $z_{ij} \in R$ and $z_i \in M$, and furthermore, $z_i \in A$.  Observe that $z_{ij}$ has no neighbors in $\{x,u,v,w\}$, as this would contradict Lemma \ref{le2}.  Thus, if the three neighbors of $z_{ij}$ other than $z_i$ are in $M$, then are all in $V_F$ and are incident with pre-colored edges under $\psi$.  However, this implies that $z_{ij}$ has two neighbors in $\{a,a_1, a_2, a_3\}$ for some $a \in \{u, v,w \}$, which contradicts Lemma \ref{le2}.  This proves (5).

If $z_i \neq w_1$ and $z_iz_{ij}, z_iz_{ik} \in E(M,R)$, then $z_i \in M$ and $z_{ij}, z_{ik} \in R$ by (5).  Suppose that $z_{ij}$ and $z_{ik}$ each have two neighbors other than $z_i$ in $M$.  By Lemma \ref{le2}, $z_{ij}$ and $z_{ik}$ have four distinct neighbors other than $z_i$ in $M$, and furthermore, none of these four vertices are in $\{x,u,v,w\}$.  Hence they must be in $A$.  Since $z_i \neq w_1$, we may assume without loss of generality that $z_i = u_i$.  By Lemma \ref{le2}, neither $z_{ij}$ nor $z_{ik}$ can have a neighbor in $\{u_1, u_2, u_3\}$ other than $u_i$.  Thus, the four neighbors previously described must be $v_1, v_2, v_3, w_1$, which contradicts Lemma \ref{le2}.   This proves (6).
 \end{proof}

\section{How to color the vertices in $L$ and $R$ `collaboratively'} \label{sec:proof}

In this section, we prove Theorem \ref{thm:main}.  Before doing so, we first prove some lemmas that show $M \cap A \neq \emptyset$ and potential properties of the vertices in $M \cap A$.  In each of the following lemmas, we aim to color $E(G[L])$ and $E(G[R])$ and order the edges incident with $M$ so that each edge $e$ has at most $20$ different colors in $N(e)$, which leads to a strong edge-coloring of $G$.  We also remove the colors placed on the edges of $G$ by $\psi$ so that $G$ is completely uncolored.

\begin{Lemma}
\label{le40.1}
There is no vertex $z\in \{u,v,w\}$ such that $z_i\in L$, $z_j\in R$ and $z_k\in L\cup R$ for $i, j, k\in [3]$. In particular, $w_1\not\in L$.
\end{Lemma}

\begin{proof}
Suppose otherwise that for some $z\in \{u,v,w\}$,  $z_1\in L$ and $z_3\in R$.  So $zz_1 \in F$.  By Lemma \ref{le6}(3)-(4),  for each $j\in [3]$, $z_1z_{1j}\in E(G[L])$, $z_3z_{3j}\in E(G[R])$ and $yy_j\in E(G[R])$.
By Lemma \ref{le21}, $|F|\geq4$. So, there are at least three edges different from $zz_1$ in $F$. Assume that $aa'$ is such an edge where $a\in V_F$ and $a'\in V_F'$.
Consider two graphs $G_L$ and $G_R$ as follows:
\begin{center} $V(G_L)=L$ and $E(G_L)=E(G[L])\cup\{z_1a'\}$;\end{center}
\begin{center} $V(G_R)=R$ and $E(G_R)=E(G[R])\cup\{z_3y\}$. \end{center}

Note that if $z_1a'$ already exists, then we add a parallel edge with endpoints $z_1$ and $a'$.
Recall that $x, z \in M$ so that $z_1$ has at most three neighbors in $L$,  $y$ and $z_3$ have at most three neighbors in $R$. Thus, $G_L$ and $G_R$ both have maximum degree at most 4.
By the minimality of $G$, both $G_L$ and $G_R$ have strong edge-colorings with 21 colors.
In $G_L$,  let the colors of the three edges incident with $z_1$ in $G[L]$ (other than the new $z_1a'$) be $1,2,3$, and the color of potentially new $z_1a'$ be $d$, respectively.  In $G_R$,  by renaming colors, let the colors of the three edges incident with $z_3$ in $G[R]$ be $1,2,3$ and the color of $yy_1$ be $d$, respectively.

We now color the edges in $G$ by giving the edges in $G[L]$ and $G[R]$ the same colors as in $G_L$ and $G_R$.  As observed before Lemma \ref{le6}, this yields a partial strong edge-coloring, which we will call $\phi$.  Thus, the edges uncolored by $\phi$ are exactly those in $F \cup E(G[M]) \cup E(M,R)$.  In particular, these are the edges incident with vertices in $M$, and recall that $M \subseteq A \cup \{x,u,v,w\}$.  Observe that the edges incident with $u, v, w$, and $x$ are all uncolored.

We now extend $\phi$ to some of the uncolored edges.  For $z'\in \{u,v,w\}-z$ and $i,j \in \{1,2,3\}$, assign $z'_iz'_{ij}$ with an available color if it is not colored yet, and assign $z'z'_i$ an available color.  This can be done as each of the aforementioned edges sees at least four uncolored edges.  This yields a new, partial strong edge-coloring, which we will call $\rho$. Observe that the edges incident with $z_2$ other than $zz_2$ are the colored edges under $\rho$.  Recall also that the edges incident with $z_1$ other than $zz_1$, and the edges incident with $z_3$ other than $zz_3$, are colored with 1,2, and 3.  Also, $yy_1$ is colored with $d$.

We finally color the remaining edges based on whether or not $d$ occurs on an edge incident with $z_2$.  Let $\{u,v,w\}-z=\{z',z''\}$.

 \begin{itemize}
\item If $d$ occurs at an edge incident with $z_2$, then color the remaining edges in the following order:
$xz',\ xz'', \ xy, \ zz_1,\ zz_3, \ zz_2, \ xz.$

\item If $d$ does not occur at the edges incident with $z_2$ in $G[R]$, then color $zz_1$ with $d$, and color the remaining edges in the following order:
$xz',\ xz'', \ xy, \ zz_3, \ zz_2, \ xz.$
 \end{itemize}

Note that in each case we always have a color available on the edges in the above sequence.  In particular, $xw$ will see four pairs of edges colored with 1,2,3, and $d$. 
Thus, $G$ has a strong edge-coloring with $21$ colors, a contradiction.
\end{proof}

\begin{Lemma}
\label{le40.2}
$M \cap A \neq \emptyset$.
\end{Lemma}

\begin{proof}
Suppose otherwise $M \cap A = \emptyset$. Then the vertices of $A$ must be partitioned amongst $L$ and $R$, and furthermore $V_F \subseteq \{u,v,w\}$.    By Lemma \ref{le40.1}, $w_1\in R$, and for each $z\in \{u,v\}$, $z_1,z_2,z_3\in L$ or $z_1,z_2,z_3\in R$. Thus, $F$ contains all or none of edges in $\{zz_1, zz_2, zz_3\}$. Note that $F$ is an edge-cut and $ww_1, ww_2, ww_3\in E(M,R)$.  This implies that $V_F \subseteq\{u,v\}$, and additionally, $F \subseteq \{zz_i: z \in \{u,v\}, i \in \{1,2,3\}\}$.  However, this implies that $\{xu, xv\}$ is also an edge-cut, contrary to Lemma \ref{le21}.
\end{proof}

{\bf Remark 1:}  For $z\in \{u,v,w\}$, if $z_i\in M$ (and so is in $V_F$), then by Lemma \ref{le6}(1), $z_i$ is incident with an edge in $F$ and an edge in $E(M,R)$;  we may assume, as a convention, that $z_iz_{i1}\in F$ and $z_iz_{i3}\in E(M,R)$.

\begin{Lemma}
\label{le8}
There exists some vertex  $z_i\in M\cap A$ such that at least three of the eight edges incident with $z_{i2}$ and $z_{i3}$ are in $E(G[R])$.
\end{Lemma}

\begin{proof}
Suppose otherwise that for each vertex $z_i\in M\cap A$, at most two of the eight  edges incident with $z_{i2}$ and $z_{i3}$ are in $G[R]$.

\smallskip
{\bf Case 1.} $w_1\in M$ and there is only one edge incident with $w_{13}$ in $G[R]$.

In this case,  since $w_{13}\in R$, the other three edges incident with $w_{13}$ must be in $E(M,R)$.  In particular, $w_{13}$ has at least two neighbors in $M \cap A$ other than $w_1$.   By  Lemma~\ref{le2}, $w_{13}$ can be adjacent to at most one vertex in each of $\{u_1, u_2, u_3\}$ and $\{v_1, v_2, v_3\}$. Without loss of generality, we may assume that $w_{13}$ is adjacent to $u_1$ and $v_1$. Then $u_1, v_1\in M$. So, $u_1u_{11},\ v_1v_{11},\ w_1w_{11}\in F$, $u_1u_{13}, v_1v_{13}, w_1w_{13}\in E(M,R)$, where $u_{13}=v_{13}=w_{13}$.

Since $w_1w_{11}\in F$, by Lemma \ref{le6}(2), the three edges incident with $w_{11}$ (other than $w_1w_{11}$) are in $G[L]$. Similarly, there are three edges incident with $u_{11}$ (other than $u_1u_{11}$) that are in $G[L]$.  In particular, $u_{11}$ is not adjacent to either $w_{11}$ or $w_{12}$, as this would contradict Lemma \ref{le2}.  In addition, there exists $u_{11}u' \in E(G[L])$ where $u' \notin \{w_{11}, w_{12}\}$.

Let $G_L$ and $G_R$ be the following graphs:

\begin{center} $V(G_L)=L$ and $E(G_L)=E(G[L])\cup \{w_{11}u_{11}\}$; \end{center}
\begin{center} $V(G_R)=R$ and $E(G_R)=E(G[R])\cup\{w_{13}w_2,w_{13}w_3,w_{13}y\}$.\end{center}

Observe that $\Delta(G_L)$ and $\Delta(G_R)$ are both at most four.
By the minimality of $G$, each of $G_L$ and $G_R$ has a strong edge-coloring with 21 colors.  In $G_L$, let the colors of the three edges incident with $w_{11}$ in $G[L]$ be $1,2,3$, respectively, and the color of $u_{11}u'$ be $d$.    In $G_R$, let the color of the edge incident with $w_{13}$ in $G[R]$ be $1$, the color of $w_{13}w_2$ be $2$,  the color of $w_{13}w_3$  be $3$, and the color of $w_{13}y$ be $d$.  Clearly, $d\notin\{ 1,2,3\}.$

We now color the edges of $G$ by assigning the edges in $G[L]$ and $G[R]$ the same colors as in $G_L$ and $G_R$, respectively.   Observe that this yields a partial strong edge-coloring of $G$ in which the only uncolored edges are incident with vertices in $M \subseteq A \cup \{x,u,v,w\}$.  Recall that $u_1, v_1, w_1 \in M$.

 Since $yw_{13}, w_2w_{13}$, and $w_3w_{13}$ are colored with $d, 2$, and 3, respectively in $G_R$, we color $xy, ww_2, ww_3$ with $d, 2, 3$, respectively.

\begin{itemize}
\item If some edge incident with $w_{12}$ has been colored with $d$, then we first color the edges $u_2u_{2j}, u_3u_{3j}, v_2v_{2j}, v_3v_{3j}$ where $j\in [3]$ (if they are not colored) with available colors, and color the remaining edges in the following order:
\begin{align*}u_1u_{11}, \ u_1u_{12}, \ uu_1, \ uu_2, \ uu_3, \ xu, \ v_1v_{11}, \ v_1v_{12}, \  vv_2, \ vv_3,\\ xv,\ vv_1,\ v_1v_{13},\ u_1u_{13},\ xw,\ w_1w_{11},\ w_1w_{12},\ w_1w_{13},\ ww_1.
\end{align*}

 \item If the edges incident with $w_{12}$ are not colored with $d$ (including the case that they are not colored), then color $w_1w_{13}$ with $d$,   color the edges $u_2u_{2j}, u_3u_{3j}, v_2v_{2j}, v_3v_{3j}$ where $j\in [3]$ (if they are not colored) with available colors, and color the remaining edges in the order (recall that $u_{11}u'$ is colored with $d$):
\begin{align*}
u_1u_{11},\ u_1u_{12}, \  uu_1, \ uu_2, \ uu_3,\ xu, \ v_1v_{11}, \ v_1v_{12}, \ vv_2, \\ vv_3,\ xv, \ vv_1,\  v_1v_{13},\  u_1u_{13},\  xw,\  w_1w_{11},\  w_1w_{12}, \ ww_1.
\end{align*}
 \end{itemize}

So, $G$ has a strong edge-coloring with $21$ colors. It is a contradiction.

 \medskip
{\bf Case 2.} $w_1\in M$ and there are exactly two edges incident with $w_{13}$ in $G[R]$.

Recall from Remark 1, that $w_1w_{11} \in F$ so that $w_{11} \in L$, and $w_1w_{13} \in E(M,R)$.  If $w_{12} \in G - L$, then it must be in $R$, and by Lemma \ref{le6}(5), $w_{12}$ would have an edge incident with it in $G[R]$.  Yet, we are assuming that at most two of the eight edges incident with $w_{12}$ and $w_{13}$ are in $G[R]$, a contradiction.  So $w_{12} \in L$ and $w_1w_{12} \in F$.

Since $w_{13}$ has exactly two neighbors in $R$, we may assume without loss of generality, that $w_{13}$ is adjacent to $u_1$.  Then $u_1\in M$, and consequently, $u_1u_{11}\in F$ and $u_1u_{13}\in E(M,R)$, where $u_{13}=w_{13}$.  Consider two graphs $G_L$ and $G_R$ as follows:

\begin{center} $V(G_L)=L$ and $E(G_L)=E(G[L])\cup\{w_{11}w_{12}\}$; \end{center}
\begin{center} $V(G_R)=R$ and $E(G_R)=E(G[R])\cup\{w_{13}w_2, w_{13}w_3\}$.\end{center}

Notice that $w_{11},w_{12}\in L$, $w_1\in M$, $w_2,w_3\in R$ and $w_{13}$ has exactly two neighbors in $R$. So, both graphs $G_L$ and $G_R$ have maximum degree at most four.
  By the minimality of $G$, both $G_L$ and $G_R$ have strong edge-colorings with 21 colors.  In $G_L$, let the colors of the three edges incident with $w_{11}$, other than $w_{11}w_{12}$, be $1,2,3$, and let the color of one edge incident with $w_{12}$, other than $w_{11}w_{12}$, be $d$ (these edges exist by Lemma \ref{le6}(2)). Clearly, $d\neq1,2,3$. In $G_R$, by renaming colors, let the color of the edges incident with $w_{13}$, other than $w_{13}w_2$ and $w_{13}w_3$, be 1 and 2, let  the color of $w_{13}w_2$ be $3$, and let the color of $w_3w_{31}$ be $d$.

We now color the edges of $G$ by assigning the edges in $G[L]$ and $G[R]$ the same colors as in $G_L$ and $G_R$, respectively.  Observe that this yields a partial strong edge-coloring of $G$ in which the only uncolored edges are incident with vertices in $M \subseteq A \cup \{x,u,v,w\}$.  Recall that $u_1, w_1 \in M$.

  Next we color $ww_2$ with $3$.
We assign $u_ju_{jk}, v_jv_{jk}$ for $j,k\in[3]$, if not colored yet, with available colors except for $u_1u_{13}$, and assign $uu_i, vv_i$ for $i\in[3]$ with available colors.  Finally, we color the remaining edges in the order:
$$xv,\ xy,\ xu,\   u_1u_{13},\ xw,\ ww_3, \  w_1w_{11}, w_1w_{12}, w_1w_{13},\ ww_1.$$
So, $G$ has a strong edge-coloring with $21$ colors. It is a contradiction.

 \medskip
{\bf Case 3.} $w_1\notin M$.  By the symmetry of $u$ and $v$, we may assume that $u_1  \in M\cap A$.

Since $u_1u_{11}\in F$ and $u_1u_{13}\in E(M,R)$, by Lemma \ref{le6}(6), $u_1u_{12}$ must be in $F$.  By Lemma \ref{le6}(2),  the three edges incident with $u_{11}$ other than $u_1u_{11}$ and the three edges incident with $u_{12}$ other than $u_1u_{12}$ are in $G[L]$.   Since $u_1u_{13}\in E(M,R)$, by Lemma \ref{le6}(5), at least one edge incident with $u_{13}$ is in $G[R]$.

Since $w_1\notin M$,  Lemma~\ref{le40.1} implies that $w_1\in R$.  By Lemma~\ref{le6}(4), $w_1w_{1j}\in G[R]$ for $j\in [3]$. Since we are assuming that at most two of the eight edges incident with $u_{12}$ and $u_{13}$ are in $G[R]$, $u_{13}$ must have a neighbor other than $u_1$ in $M$.  Since $w_1 \notin M$, we may assume it is $v_1$ so that $v_1 \in M$.  Note that by Lemma \ref{le2}, $u_{13}$ cannot have any other neighbors in $M$, as they would be in $\{u_2, u_3, v_2, v_3\}$.   Thus, there are exactly two edges incident with $u_{13}$ in $G[R]$.


 By a similar argument to the above, $u_1u_{11}, u_1u_{12}, v_1v_{11}, v_1v_{12}\in F$, $u_1u_{13}, v_1v_{13}\in E(M,R)$, and $u_{13}=v_{13}$. By Lemma \ref{le2}, $u_{11}, u_{12}, v_{11}, v_{12}$ are all distinct.

 Consider two graphs $G_L$ and $G_R$ as follows:

\begin{center} $V(G_L)=L$ and $E(G_L)=E(G[L])\cup\{u_{11}u_{12}\}$; \end{center}
\begin{center} $V(G_R)=R\cup \{b\}$ and $E(G_R)=E(G[R])\cup\{w_1b, w_2b, w_3b, u_{13}b, u_{13}y\}$.\end{center}

Notice that $x, w, u_1\in M$, $w_1, w_2, w_3, y\in R$ and $u_{13}$ has exactly two neighbors in $R$. So, both graphs $G_L$ and $G_R$ have maximum degree at most four.
 By the minimality of $G$, both $G_L$ and $G_R$ have strong edge-colorings with 21 colors.  In $G_L$, let the colors of the three edges incident with $u_{11}$  (other than $u_{11}u_{12}$) be $1,2,3$, and  the color of $u_{11}u_{12}$ be $d$. Clearly, $d\neq1,2,3.$  In $G_R$, by renaming colors, let the colors of the two edges incident with $u_{13}$ (other than $u_{13}b, u_{13}y$) be $1,2$, the color of $u_{13}y$ be $3$, the color of $u_{13}b$ be $d$, and the colors of $w_1b, w_2b, w_3b$ be $d_1,d_2,d_3$, respectively. Clearly, $\{d_1,d_2,d_3\} \cap \{1,2,3,d\} = \emptyset$.

\medskip

{\bf Claim:} the colors $1,2,3$ appear on edges incident with $v_{11}$ or $v_{12}$ in $G_L$.
\begin{quote}
\begin{proof}
 Suppose otherwise that  at least one of colors $1,2,3$ does not appear.  If 3 appears on an edge incident with $v_{11}$ or $v_{12}$ but 1 does not, then switch the colors 3 and 1 in $G_L$ so that 3 is missing.  We do a similar switch if 3 appears, but 2 does not.  Thus, we may assume that 3 does not appear on  edges incident with $v_{11}$ or $v_{12}$.

We now color the edges of $G$ by assigning the edges in $G[L]$ and $G[R]$ the colors used in $G_L$ and $G_R$, respectively.  Note that $v_{13}y\notin E(G)$, by Lemma \ref{le2}. So we can color $v_1v_{13}$ and $xy$ with $3$. We next color $ww_1, ww_2, ww_3$ with $d_1,d_2,d_3$, respectively.  We assign $u_2u_{2j}, u_3u_{3j}, v_2v_{2j}, v_3v_{3j}$ for $j\in[3]$ with available colors if they are not colored yet.  Finally,  we color the remaining edges in the  order:
$$v_1v_{11},\ v_1v_{12},\ vv_1,\ vv_2,\ vv_3,\ xv,\ xw,\ xu,\ uu_2,\ uu_3,\ u_1u_{11},\ u_1u_{12},\ u_1u_{13},\ uu_1.$$

So, $G$ has a strong edge-coloring with $21$ colors. It is a contradiction.
 \end{proof}
 \end{quote}

\medskip

We now color the edges of $G$ by assigning the edges in $G[L]$ and $G[R]$ the same colors as in $G_L$ and $G_R$, respectively.  Observe that this yields a partial strong edge-coloring of $G$ in which the only uncolored edges are incident with vertices in $M \subseteq A \cup \{x,u,v,w\}$.  Recall that $u_1, v_1 \in M$.
We next color $xy$ with $3$, color $xw$ and $u_1u_{13}$ with $d$, and color $ww_1, ww_2, ww_3$ with $d_1,d_2,d_3$, respectively. We assign $u_2u_{2j}, u_3u_{3j}, v_2v_{2j}, v_3v_{3j}$ for $j\in[3]$ with available colors if they are not colored yet.  Finally, we color the remaining edges in the order:
$$uu_2,\ uu_3, \ u_1u_{11}, \ u_1u_{12},\ xu,\ xv,\ vv_2,\ vv_3, \ v_1v_{11},\ v_1v_{12},\ v_1v_{13}, \ vv_1,\ uu_1.$$
With the claim, it is easy to check that each edge has an available color. So, $G$ has a strong edge-coloring with $21$ colors. It is a contradiction.
\end{proof}


\begin{Lemma}
\label{le9}
If there exists $z_i \in M \cap A$ such that at least three of the eight edges incident with $z_{i2}$ and $z_{i3}$ are in $G[R]$, then $z_j \notin R$ for all $j \in [3]$.  In particular, $w_1$ is not such a vertex in $M \cap A$.
\end{Lemma}


\begin{proof}
Suppose otherwise that for some $z_i \in M \cap A$, at least three of the eight edges incident with $z_{i2}$ and $z_{i3}$ are in $G[R]$ and $z_j\in R$.  Without loss of generality, we may assume that $i = 1$ and that $z_3 \in R$.   Recall that by our convention in Remark 1, $z_1z_{11}\in F$, $z_1z_{13} \in E(M,R)$, and by Lemma \ref{le6}(2), three edges incident with $z_{11}$ are in $G[L]$.

By Lemma \ref{le21}, $|F|\geq4$. It follows that at least three edges other than $z_1z_{11}$ are in $F$. Assume that $aa'$ is such an edge with $a\in V_F$ and $a'\in V_F'$. Consider the graph $G_L$:

\begin{center} $V(G_L)=L$ and $E(G_L)=E(G[L])\cup\{z_{11}a'\}$,\end{center}

where if $z_{11}a'$ already exists, then we add a parallel edge with endpoints $z_{11}$ and $a'$.
Observe that $\Delta(G_L) \le 4$.
By the minimality of $G$, $G_L$ has a strong edge-coloring with 21 colors.  In $G_L$, let the colors of the three edges incident with $z_{11}$ (other than the new copy of $z_{11}a'$) be $1,2,3$, and the color of new copy of $z_{11}a'$ be $d$, respectively.

 We may assume that either $z_{13}$ is incident with three edges in $G[R]$, or both $z_{12}$ and $z_{13}$ are incident with at most two edges in $G[R]$.   Consider $G_R$ with $V(G_R)=R$ and
 \begin{center}
 $E(G_R)=\begin{cases}
 		E(G[R])\cup\{z_{13}z_3\}, &\text{ if $z_{13}$ is incident with three edges in $G[R]$;}\\
		E(G[R])\cup\{z_{13}z_{12}, z_{13}z_3\}, & \text{ otherwise.}
		\end{cases}$
\end{center}

Observe that $\Delta(G_R) \le 4$.
By the minimality of $G$, $G_R$ have a strong edge-coloring with $21$ colors. In $G_R$,  let the colors of any three edges in $G[R]$ incident with $z_{12}, z_{13}$ be $1,2,3$, respectively. By Lemma \ref{le6}(4), $z_3$ is incident with three edges in $G[R]$.  So we may assume that one of them, say $z_3z_{31}$, is colored with $d$ (up to renaming it), which is possible even if $z_{12}$ is incident with less than three edges in $G[R]$. 

Now we color the edges in $G$.  First of all, the edges in $G[L]$ and $G[R]$ keep their colors in $G_L$ and $G_R$.  For $z'\in \{u,v, w\}-z$, we assign $z'_iz'_{ij}$ for $i,j\in[3]$ with an available color if it is not colored yet, and then assign $z'z'_i$ for $i\in [3]$ with an available color. We then color the edges $z_2z_{2j}$ for $j \in [3]$ with an available color (note that the edges $z_3z_{3j}$ are colored).  Finally, we color the remaining edges according to whether the color $d$ appears on the edges incident with $z_{12}$ (let $\{u,v,w\}-z=\{z', z''\}$):
\begin{itemize}
\item If the color $d$ does not appear at the edges incident with $z_{12}$, then color $z_1z_{11}$ with $d$, and color the remaining edges in the following order
    $$xz', \ xz'', \ xy, \ xz, \ zz_2, \ zz_3, \ z_1z_{12},\  z_ 1z_{13},\ zz _1.$$

\item If the color $d$ appears at an edge incident with $z_{12}$, then color the remaining edges in the following order:
 $$xz', \ xz'', \ xy, \ xz,  \  zz_2, \ zz_3,\ z_1z_{11},\ z_1z_{12},\ z_1z_{13},\ zz_1.$$
\end{itemize}
 So, $G$ has a strong edge-coloring with $21$ colors. It is a contradiction.
\end{proof}


We are now ready to finish the proof of Theorem~\ref{thm:main}.

\begin{proof}[Proof of Theorem \ref{thm:main}.] By Lemmas \ref{le8} and \ref{le9}, there exists some vertex $z_i\in (M \cap A)\setminus\{w_1\}$ such that at least three of the eight edges incident with $z_{i2}$ and $z_{i3}$ are in $G[R]$.  Without loss of generality, we will assume $z_i = u_1$.  Recall that from Remark 1, we will assume $u_1u_{11} \in F$ and $u_1u_{13} \in E(M,R)$.  Thus, by Lemma \ref{le6}(2) and (5) three edges incident with $u_{11}$ are in $G[L]$ and there is at least one edge incident with $u_{13}$ in $G[R]$.  Note that by Lemma~\ref{le9}, $u_2,u_3\notin R$. So, we consider the following cases.

 \medskip
 {\bf Case 1.} $u_2\in L$ or $u_3\in L$.  Without loss of generality, let $u_2\in L$.

By Lemma \ref{le6}(3), $u_2u_{2j}$ is in $G[L]$ for each $j \in [3]$.
 We consider graph $G_L$:

\begin{center}  $V(G_L)=L$ and $E(G_L)=E(G[L])\cup\{u_{11}u_2\}$.\end{center}

Observe that $\Delta (G_L)\leq4$. By the minimality of $G$, $G_L$ has a strong edge-coloring with 21 colors.  In $G_L$, let the colors of the three edges in $G[L]$ incident with $u_{11}$ be $1,2,3$, and the color of $u_2u_{21}$ be $d$, respectively.

We may assume that either $u_{13}$ is incident with three edges in $G[R]$, or both $u_{12}$ and $u_{13}$ are incident with at most two edges in $G[R]$.   Consider $G_R$ with $V(G_R)=R$ and
 \begin{center}
 $E(G_R)=\begin{cases}
 		E(G[R])\cup\{u_{13}y\}, &\text{ if $u_{13}$ is incident with three edges in $G[R]$;}\\
		E(G[R])\cup\{u_{13}u_{12}, u_{13}y\}, & \text{ otherwise.}
		\end{cases}$
\end{center}

Observe that $\Delta (G_R)\leq4$. By the minimality of $G$, $G_R$ has a strong edge-coloring with 21 colors. In $G_R$, by renaming colors, let the colors of (any) three edges incident with $u_{13}, u_{12}$ in $G[R]$ be $1,2,3$, and the color of $u_{13}y$ be $d$, respectively.

Now we color the edges in $G$, where the edges in $G[L]$ and $G[R]$ keep their colors in $G_L$ and $G_R$. 
 We then assign $v_iv_{ij}, w_iw_{ij}, u_3u_{3j}$ for $i,j\in[3]$ with an available color if it is not colored yet, and assign $vv_i, ww_i$ for $i\in [3]$ with available colors.    Finally, we color the remaining edges according to whether the color $d$ appears on the edges incident with $u_{12}$:
\begin{itemize}
\item If the color $d$ does not appear at the edges incident with $u_{12}$, then color $u_1u_{13}$ with $d$, and color the remaining edges in the following order:
$$xv, \ xw, \ xy, \ xu, \ uu_2, \ uu_3, \ u_1u_{11}, \ u_1u_{12}, \ uu_1.$$
\item  If the color $d$ appears at an edge incident with $u_{12}$, then color the remaining edges in the following order:
$$xv, \ xw, \ xy, \ xu, \ uu_2, \ uu_3, \ u_1u_{11}, \ u_1u_{12}, \ u_1u_{13}, \ uu_1.$$
\end{itemize}
In either case, $G$ has a strong edge-coloring with $21$ colors. It is a contradiction.
 \medskip

 {\bf Case 2.} $u_2,u_3\in M$. 

By Lemma~\ref{le6},  we have $u_1u_{11},u_2u_{21},u_3u_{31}\in F$ and $u_1u_{13},u_2u_{23},u_3u_{33}\in E(M,R)$.

{\bf Case 2.1.}  For some $i,j\in [3]$ with $i\not=j$, $u_iu_{i2}\in F$ and $u_ju_{j2}\in E(M,R)$.

 Assume that $u_1u_{12}\in F$ and $u_2u_{22}\in E(M,R)$. Consider two graphs $G_L$ and $G_R$ as follows:

\begin{center} $V(G_L)=L\cup \{a\}$ and $E(G_L)=E(G[L])\cup\{u_{11}a,u_{12}a,u_{21}a,u_{31}a\}$;\end{center}
\begin{center} $V(G_R)=R\cup \{b\}$ and $E(G_R)=E(G[R])\cup\{u_{13}b,u_{22}b,u_{23}b,u_{33}b\}$.\end{center}

Observe that $\Delta (G_L)\leq4$ and $\Delta (G_R)\leq4$. By the minimality of $G$, both $G_L$ and $G_R$ can be colored with 21 colors. Let the colors of $au_{11},au_{12},au_{21}$ be $1,2,3$, respectively.  We rename colors of edges in $G_R$ so that the colors of $bu_{23}, bu_{22}, bu_{13}$ are $1,2,3$, respectively.  We further assume that an edge incident with $u_{31}$ (other than $u_{31}a$) and an edge incident with $u_{33}$ (other than $u_{33}b$) have the same color, say $d$.  Clearly, $d\not=1,2,3$.

Now we color the edges in $G$.  First of all, the edges in $G[L]$ and $G[R]$ keep their colors in $G_L$ and $G_R$.   Then we color $u_1u_{11}$ and $u_2u_{23}$ with $1$,  color $u_1u_{12}$ and $u_2u_{22}$ with $2$, and color $u_1u_{13}$ and $u_2u_{21}$ with $3$.  We assign $v_iv_{ij}, w_iw_{ij}$ for $i,j\in[3]$ with available colors if they are not colored yet, and assign $vv_i, ww_i$ for $i\in [3]$ with available colors.  Finally,  we color the remaining edges in the  order:
$$xv,\ xw,\ xy,\ xu,\ u_3u_{31},\ u_3u_{32},\ u_3u_{33},\ uu_1,\ uu_2,\  uu_3.$$

So, $G$ has a strong edge-coloring with $21$ colors. It is a contradiction.


 \medskip

{\bf Case 2.2.} $u_1u_{12},u_2u_{22}\in F$, or $u_1u_{12},u_2u_{22}\in E(M,R)$.

If $u_1u_{12},u_2u_{22}\in F$,  then clearly, the three edges incident with $u_{13}$ other than $u_1u_{13}$ are in $G[R]$.  Consider two graphs $G_L$ and $G_R$ as follows:

\begin{center} $V(G_L)=L\cup \{a\}$ and $E(G_L)=E(G[L])\cup\{u_{11}a, u_{12}a,u_{21}a,u_{22}a\}$;\end{center}
\begin{center}$V(G_R)=R$ and $E(G_R)=E(G[R])\cup\{u_{13}u_{23}\}$. \end{center}

Observe that $\Delta (G_L)\leq4$ and $\Delta (G_R)\leq4$. By the minimality of $G$, both $G_L$ and $G_R$ have strong edge-colorings with 21 colors. In $G_L$,  let the colors of the three edges incident with $u_{11}$ in $G[L]$ be $1,2,3$, and the color of $u_{11}a$ be $d$, respectively. In $G_R$, by renaming the colors of edges in $G_R$,  let the colors of the three edges incident with $u_{13}$ in $G[R]$ be $1,2,3$, and the color of $u_{13}u_{23}$ be $d$, respectively.

If $u_1u_{12},u_2u_{22}\in E(M,R)$,  consider $G_L$ with $V(G_L)=L$ and $G_R$ with $V(G_R)=R\cup\{b\}$ and

\begin{center} $E(G_L)=E(G[L])\cup\{u_{11}u_{21}\}$;\end{center}

\begin{center}
$E(G_R)=\begin{cases}
	E(G[R])\cup\{u_{12}b,u_{13}b,u_{22}b,u_{23}b\},  \text{ if $u_{13}$ is incident with three edges in $G[R]$;}\\
	E(G[R])\cup\{u_{12}b,u_{13}b,u_{22}b,u_{23}b,u_{12}u_{13}\},  \text{ otherwise.}
	\end{cases}$
	\end{center}

Observe that $\Delta (G_L)\leq4$ and $\Delta (G_R)\leq4$. By the minimality of $G$, both $G_L$ and $G_R$ have strong edge-colorings with 21 colors. In $G_L$,  let the colors of the three edges incident with $u_{11}$ in $G[L]$ be $1,2,3$, and the color of $u_{11}u_{21}$ be $d$, respectively. In $G_R$, let the colors on any three edges in $G[R]$ incident with $u_{12},u_{13}$ be $1,2,3$, and the color of $u_{23}b$ be $d$, respectively.

In either case,  we color the edges in $G$ in the following procedure.  First of all, the edges in $G[L]$ and $G[R]$ keep their colors in $G_L$ and $G_R$. Next, we color $u_1u_{11}$ and $u_2u_{23}$ with $d$, and  assign $v_iv_{ij}, w_iw_{ij}$ for $i,j\in[3]$ with available colors if they are not colored yet, and assign $vv_i, ww_i$ for $i\in [3]$ with available colors. Finally, we color the remaining edges in the following order:
$$xv, xw, xy,\  u_2u_{21},\  u_2u_{22},\  u_3u_{31},\ u_3u_{32},\ u_3u_{33},\  xu,\  uu_2,\  uu_3,\ u_1u_{12},\  u_1u_{13},  uu_1.$$

So, $G$ has a strong edge-coloring with $21$ colors. It is a contradiction.
\end{proof}

\section{Proof of Lemma~\ref{le2}}\label{proof-le2}

In this section, we proof Lemma~\ref{le2} in a series of lemmas.  In these proofs, we will often remove vertices and edges from $G$ to obtain a strong edge-coloring, say $\phi$, of the remaining graph that use at most 21 colors. Often, we will consider $|A_\phi(e)|$ for each uncolored edge $e$ of $G$ with the purpose of applying the well-known result of Hall \cite{Hall} in terms of systems of distinct representatives.  This yields a coloring of the remaining uncolored edges such that they will receive distinct colors, which ultimately produces a strong edge-coloring of $G$.  Thus, when in a situation in which we can apply this result of Hall, we will say that we obtain a strong edge-coloring of $G$ by SDR.

Let $\sU_\phi(v)$ to be the set of colors used on edges incident with a vertex $v$.  For adjacent vertices $u$ and $v$, let $\Upsilon_\phi(u,v)$ be $\sU_\phi(u)\setminus\{\phi(uv)\}$.  That is, $\Upsilon_\phi(u,v)$ is the set of colors used on the edges incident with $u$ other than $uv$.  Observe that $\Upsilon_\phi(u,v)$ and $\Upsilon_\phi(v,u)$ are disjoint.  Often, we will refer to only one partial strong edge-coloring that will not be named.  In such cases we will suppress the subscripts used in the above notations.

\begin{lemma}\label{lem:simple}
$G$ has no multiple edges.  That is, $G$ is simple.
\end{lemma}

\begin{proof}
Suppose on the contrary that there exists a parallel edge $e$ with endpoints $u,v$.  By the minimality of $G$, $G - e$ has a good coloring.  Since $e$ has at least five colors available, we can obtain a good coloring of $G$.
\end{proof}

\begin{lemma}\label{lem:NoTriangle}
$G$ contains no triangles.
\end{lemma}

\begin{proof}
Suppose on the contrary that $G$ contains a triangle $u_1,u_2,u_3$.  Since $G$ is 4-regular, there exist $x_i,y_i \in N(u_i)\setminus\{u_1,u_2,u_3\}$.  By the minimality of $G$, $G - \{u_1,u_2,u_3\}$ has a good coloring.

Observe that $|A(x_iu_i)|, |A(y_iu_i)| \ge 6$ for $i \in [3]$ and $|A(u_ju_{j+1})| \ge 9$ for $j \in [3]$ modulo 3.  Thus, we obtain a good coloring of $G$ by SDR.
\end{proof}

\begin{lemma}\label{lem:NoK_3,3}
$G$ contains no $K_{3,3}$.
\end{lemma}

\begin{proof}
Suppose on the contrary that $G$ contains $K_{3,3}$ as a subgraph with partite sets $\{u_1,u_2, u_3\}$ and $\{v_1,v_2,v_3\}$.  For $i \in [3]$, let $x_i$ denote the fourth neighbor of $u_i$ not in $\{v_1,v_2,v_3\}$, and let $y_i$ denote the fourth neighbor of $v_i$ not in $\{u_1,u_2,u_3\}$.

Let $G'$ be obtained from $G$ by removing $u_1,u_2,u_3,v_1,v_2,v_3$.  By the minimality of $G$, $G'$ has a good coloring that we can impose onto $G$.  Observe that $|A(u_ix_i)|, |A(v_iy_i)| \ge 9$ for $i \in [3]$, and $|A(u_jv_\ell)| \ge 15$ for $1 \le j \le \ell \le 3$.  We then obtain a good coloring of $G$ by SDR.
\end{proof}

\begin{lemma}\label{lem:NoK_2,4}
$G$ contains no $K_{2,4}$.
\end{lemma}

\begin{proof}
Suppose on the contrary that $G$ contains $K_{2,4}$ as a subgraph with partite sets $\{u_1,u_2,u_3,u_4\}$ and $\{v_1,v_2\}$.  For $i \in [4]$, let $x_i,y_i$ denote the third and fourth neighbors of $u_i$ not in $\{v_1,v_2\}$.  Of course, $x_i \neq y_i$, and by Lemma \ref{lem:NoTriangle}, $x_i, y_i \notin \{u_1,u_2,u_3,u_4\}$ for $i \in [4]$.  So $|\{x_1,y_1,\dots, x_4,y_4\}| \ge 4$.

Let $G'$ be obtained from $G$ by removing $v_1$ and $v_2$.  By the minimality of $G$, $G'$ has a good coloring that we can impose onto $G$.  Call it $\phi$.   Note that if $e, e' \in E(G')$, and $\phi(e) = \phi(e')$, then they are still sufficiently far apart in $G$.  Thus, $\phi$ is a good partial coloring of $G$.   Observe that $|A_\phi(u_iv_j)| \ge 7$ for $i \in [4], j \in [2]$.

If $|\bigcup\limits_{i \in [4], j \in [2]} A_\phi(u_iv_j)| \ge 8$, then we can greedily color the remaining edges to obtain a good coloring of $G$.  Therefore, since $|A_\phi(u_iv_j)| \ge 7$ for each $i \in [4]$ and $j \in [2]$, we may assume each $A_\phi(u_iv_j) = [7]$.  Observe that this implies each $u_ix_i$ and $u_iy_i$ receives distinct colors.  So without loss of generality, suppose they are colored with the colors from $[15]\setminus[7]$.  Furthermore, we may assume $\Upsilon_\phi(x_i, u_i) \cup \Upsilon_\phi(y_i,u_i) = [21]\setminus[15]$, for each $i \in [4]$, as otherwise $A_\phi(u_iv_j) \neq [7]$ for some $i$ and $j$.  This also implies that $|\{x_1, y_1, \dots, x_4, y_4\}| = 8$; that is, they are all distinct.

Thus, our goal is to recolor two edges among $\{u_ix_i, u_iy_i: i \in [4]\}$ to be the same and obtain a good partial coloring of $G$.  If so, then we can color the remaining edges greedily to obtain a good coloring of $G$.  As a result, if we uncolor an edge $u_ix_i$, then in the resulting good partial coloring, the only colors available on this edge must be contained in $[7] \cup \phi(u_ix_i)$.

Note that by Lemma \ref{le1}, $x_1$ cannot be adjacent to every vertex in $\{x_2, y_2, x_3, y_3, x_4, y_4\}$.  So we may assume $x_1x_2 \notin E(G)$.  Uncolor the edges $u_1x_1$ and $u_2x_2$, and let $\sigma$ be this good partial coloring of $G$.  Since the colors on these edges in $\phi$ were distinct and in $[15]\setminus[7]$, we may assume they were 8 and 9.  Observe that $|A_\sigma(u_ix_i)| \ge 5$ for $i \in [2]$, and $A_\sigma(u_iv_j) = [9]$ for $i \in [4], j \in [2]$.

As noted, $A_\sigma(u_1x_1) \cup A_\sigma(u_2x_2) \subseteq [9]$.  However, since each edge now has at least 5 colors available, there must be some $\alpha \in A_\sigma(u_1x_1) \cup A_\sigma(u_2x_2)$.  Thus, we can color these two edges with $\alpha$ to obtain a coloring $\psi$.  Since $x_1x_2 \notin E(G)$, and since the $x_i$'s and $y_i$'s are all distinct, $\psi$ is a good partial coloring of $G$ in which $|A_\psi(u_iv_j)| \ge 8$ for $i \in [4], j \in [2]$.  Thus, we can greedily color the remaining edges to obtain a good coloring of $G$.
\end{proof}

\begin{lemma}\label{lem:NoK_2,3}
$G$ contains no $K_{2,3}$.
\end{lemma}

\begin{proof}
Suppose on the contrary that $G$ contains a $K_{2,3}$ with partite sets $\{u_1,u_2,u_3\}$ and $\{v_1,v_2\}$.  By Lemma \ref{lem:NoTriangle}, this subgraph is induced, and as $G$ is 4-regular, for $i \in [3]$, there exist vertices $x_i,y_i$ adjacent to $u_i$ other than $v_1, v_2$, and vertices $z_1,z_2$ adjacent to $u_1,u_2$, respectively, other than $u_1,u_2,u_3$.  By Lemma \ref{lem:NoK_2,4}, $z_1\neq z_2$, and by Lemma \ref{lem:NoTriangle}, $z_1,z_2$ are distinct from the $x_i,y_i$.

We define the following sets, for $j \in [2]$, let $\sY_j := \{u_iv_j: i \in [3]\}$, let $\sY := \sY_1 \cup \sY_2$,  let  $\sZ := \{v_1z_1, v_2z_2\}$, and let $\sX :=\{u_ix_i, u_iy_i: i \in [3]\}$.

We proceed based on the existence of $z_1z_2 \in E(G)$.

\begin{case}
$z_1z_2 \in E(G)$.
\end{case}

Let $G'$ be the graph obtained from $G$ by deleting $v_1$ and $v_2$.  By the minimality of $G$, $G'$ has a good coloring $\phi$.

Observe that $|A_\phi(e)| \ge 6$ for $e \in \sY$ and $|A_\phi(e')| \ge 4$ for $e' \in \sZ$.   Since $z_1z_2 \in E(G)$, we may assume $\sU_\phi(z_1) = \{1,2,5\}$ and $\sU_\phi(z_2) = \{3,4,5\}$ so that $\phi(z_1z_2) = 5$.

We can extend $\phi$ by coloring the edges of $\sZ$, and denote this good partial coloring by $\sigma$.  Note that neither edge in $\sZ$ is colored with 1 or 2.  Now, every edge in $\sY$ has at least four colors available on it.  We proceed by considering which edges incident to some $x_i$ or $y_i$ are already colored with either 1 or 2.

We first claim that neither 1 nor 2 appear on any edge of $\sX$ under $\sigma$.   If 1 and 2 both appear, then each vertex in $\sY_1$ has at least six colors available, and we can extend $\sigma$ by SDR.  If only 1 appears on an edge of $\sX$, then each edge in $\sY_1$ has at least five colors available.   So if one edge in $\sY$ has at least six colors available, we can extend $\sigma$ by SDR.  This implies that for every $i \in [3]$, $2 \notin \sU_\sigma(x_i) \cup \sU_\sigma(u_i)$, else $|A_\sigma(u_iv_1)| \ge 6$.  Thus, we can color any edge in $\sY_2$ with 2 to obtain a good partial coloring $\psi$.  Observe that $|A_\psi(e)| \ge 4$ for $e \in \sY_2$, and $|A_\psi(e')| \ge 5$ for $e' \in \sY_1$.  So we can extend $\psi$ by SDR, which proves our claim.

We now return to $\sigma$.  Suppose $1 \notin \sU_\sigma(x_1) \cup \sU_\sigma(y_1)$.  If in addition, $2 \notin \sU_\sigma(x_2) \cup \sU_\sigma(y_2)$.  Then we can color $u_1v_2$ and $u_2v_2$ with 1 and 2, respectively to obtain a good partial coloring of $G$.  From here we can color the remaining four edges by SDR as every edge in $\sY_1$ still has at least four colors available.  Then by a similar argument, we may assume $2 \in \sU_\sigma(x_i) \cup \sU_\sigma(y_i)$ for $i \in \{2,3\}$.  We again color $u_1v_2$ with 1 to obtain a good partial coloring of $G$.  From here we can color the remaining five edges by SDR as $v_1u_2$ and $v_1u_3$ each have at least five colors available.

Thus, we can assume $1, 2 \in \sU_\sigma(x_i) \cup \sU_\sigma(y_i)$ for each $i \in [3]$.  We then color the edges of $\sY$ be SDR as every edge in $\sY_1$ has at least six colors available.  This completes the case.

\begin{case}
$z_1z_2 \notin E(G)$.
\end{case}

Let $G'$ be the graph obtained from $G$ by deleting $v_1$ and $v_2$ and adding the edge $z_1z_2$.  By the minimality of $G$, $G'$ has a good coloring, which ignoring $z_1z_2$, can be applied to $G$.  We immediately extend this by coloring $v_1z_1$ and $v_2z_2$ with the color used on $z_1z_2$.    Call this good partial coloring $\phi$.

We may assume that $\sU_\phi(z_1) = \{15, 16, 17, 21\}$ and $\sU_\phi(z_2) = \{18, 19, 20, 21\}$ so that $\phi(v_1z_1) = \phi(v_2z_2) = 21$.   Observe that $|A_\phi(e)| \ge 5$ for $e \in \sY$.  If $|A_\phi(e)| \ge 6$ for any $e \in \sY$, we can color the remaining edges of $G$ by SDR.

If 15, 16, or 17 appears in some $A_\phi(x_i) \cup A_\phi(y_i)$ for some $i \in [3]$, then $|A_\phi(v_1u_i)| \ge 6$, and we are done.  So we may color the edges in $\sY_2$ with 15, 16, and 17, to obtain a good partial coloring of $G$, call it $\sigma$.  Observe that $|A_\sigma(e)| \ge 5$ for each $e \in \sY_1$, so that we can color them greedily to obtain a good coloring of $G$.

This completes the proof of the case, and hence, proves the lemma.
\end{proof}

\begin{lemma}\label{lem:No4cycle}
$G$ has no 4-cycles.
\end{lemma}

\begin{proof}
Suppose on the contrary that $u_1u_2u_3u_4$ is a 4-cycle in $G$.  By Lemmas \ref{le1} and \ref{lem:NoTriangle}, for each $i \in [4]$, there exists $x_i,y_i \in N(u_i)\setminus\{u_1,u_2,u_3,u_4\}$.  By Lemmas \ref{lem:NoTriangle} and \ref{lem:NoK_2,3}, $x_1,y_1,\dots, x_4,y_4$ are distinct.  Define the sets $\sX := \{x_iu_i, y_iu_i: i \in [4]\}$, $\sY := \{x_ix_{i+1}: i \in [4] \text{ modulo 4}\}$, $\sX_i = \{x_iu_i, y_iu_i\}$ for $i \in [4]$.

By the minimality of $G$, $G - \{u_1,u_2,u_3,u_4\}$ has a good coloring $\phi$ that we can apply to $G$.  Observe that $|A_\phi(e)| \ge 6$ for $e \in \sX$ and $|A_\phi(e')| \ge 9$ for $e' \in \sY$.  We proceed based on if we can extend $\phi$ by coloring  the edges of $\sX_1$ and $\sX_3$ (or $\sX_2$ and $\sX_4$) with the same colors.

\setcounter{case}{0}
\begin{case}
We can extend $\phi$ by coloring the edges of $\sX_1$ and $\sX_3$ with 1 and 2.
\end{case}

Suppose we can extend $\phi$ by coloring $x_1u_1, x_3u_3$ with 1, and $y_1u_1, y_3u_3$ with 2.  Call this good partial coloring $\sigma$.  Observe that $|A_\sigma(e)| \ge 4$ for $e \in \sX_2 \cup \sX_4$ and $|A_\sigma(e')| \ge 7$ for $e' \in \sY$.

If there are at least eight colors available over all the edges of $\sY$ under $\sigma$, then we can obtain a good coloring of $G$ by SDR.  Thus, $A_\sigma(e_1) = A_\sigma(e_2)$ and $|A_\sigma(e_1)| = 7$ for all $e_1,e_2 \in \sY$.  Without loss of generality, we may assume $\Upsilon_\sigma(x_i,u_i) \cup \Upsilon_\sigma(y_i,u_i) =  \{3,4,5,6,7,8\}$ for $i = 1,3$, and $\Upsilon_\sigma(x_i,u_i) \cup \Upsilon_\sigma(y_i,u_i) = \{9,10,11,12,13,14\}$ for $i = 2,4$.

Now, $x_2,y_2,x_4,y_4$ cannot induce a $K_{2,2}$ by Lemma \ref{lem:NoK_2,3}.  So, say $x_2x_4 \notin E(G)$.  If $|A_\sigma(x_2u_2) \cup A_\sigma(x_4u_4)| \ge 8$, then we obtain a good coloring of $G$ by SDR.  Thus, we can extend $\sigma$ by coloring $x_2u_2, x_4u_4$ with the same color, and then further extend by SDR.

\begin{case}
We can extend $\phi$ by coloring an edge of $\sX_1$ and an edge of $\sX_3$ with 1.
\end{case}

We may assume that we can extend $\phi$ by coloring $x_1u_1, x_3u_3$ with 1.  Suppose that we can further extend $\phi$ by coloring an edge of $\sX_2$ and $\sX_4$ with 2, say $x_2u_2, x_4u_4$.  Call this good partial coloring $\sigma$.  Observe that $|A_\sigma(e)| \ge 4$ for uncolored $e \in \sX$ and $|A_\sigma(e')| \ge 7$ for $e' \in \sY$.

As in the previous case, we may assume that $\Upsilon_\sigma(x_i,u_i) \cup \Upsilon_\sigma(y_i,u_i) =  \{3,4,5,6,7,8\}$ for $i = 1,3$, and $\Upsilon_\sigma(x_i,u_i) \cup \Upsilon_\sigma(y_i,u_i) = \{9,10,11,12,13,14\}$ for $i = 2,4$ so that $A_\sigma(e) = A_\sigma(e')$ and $|A_\sigma(e)| = 7$ for $e,e' \in \sY$.

Now, suppose $y_1y_3 \in E(G)$.  Then $|A_\sigma(y_1u_1)|, |A_\sigma(y_3u_3)| \ge 7$, and we can obtain a good coloring of $G$ by coloring $y_2u_2, y_4u_4, x_1x_2, x_2x_3, x_3x_4, x_4x_1, y_1u_1, y_3u_3$ in this order.

So $y_1y_3 \notin E(G)$, and by the previous case $A_\sigma(y_1u_1) \cap A_\sigma(y_3u_3) = \emptyset$.  Thus, $|A_\sigma(y_1u_1) \cup A_\sigma(y_3u_3)| \ge 8$, and we can obtain a good coloring of $G$ by SDR.

Thus, it remains to consider when we cannot extend $\phi$ by coloring an edge of $\sX_2$ and $\sX_4$ with a common color.  By Lemma \ref{lem:NoK_2,3}, we may assume $x_2x_4 \notin E(G)$.  Let $\psi$ denote the good partial coloring extending $\phi$ by coloring $x_1u_1, x_3u_3$ with 1.

Observe that $|A_\psi(e)| \ge 5$ for uncolored $e \in \sX$ and $|A_\psi(e')| \ge 8$ for $e' \in \sY$.  Since $x_2x_4 \notin E(G)$, we must have $|A_\psi(x_2u_2) \cup A_\psi(x_4u_4)| \ge 10$, otherwise we can color $x_2u_2, x_4u_4$ with a common color, a contradiction.  Now, if there are at least nine colors available over all the edges of $\sY$ under $\psi$, then we obtain a good coloring of $G$ by SDR.  Thus, we have $A_\psi(e_1) = A_\psi(e_2)$ and $|A_\psi(e_1)| = 8$ for $e_1, e_2 \in \sY$.

As above, we may assume $\Upsilon_\sigma(x_i,u_i) \cup \Upsilon_\sigma(y_i,u_i) =  \{2,3,4,5,6,7\}$ for $i = 1,3$, and $\Upsilon_\sigma(x_i,u_i) \cup \Upsilon_\sigma(y_i,u_i) = \{8,9,10,11,12,13\}$ for $i = 2,4$.

Suppose $y_2y_4 \notin E(G)$.  By the previous case, $A_\sigma(y_2u_2) \cap A_\sigma(y_4u_4) = \emptyset$ so that $|A_\sigma(y_2u_2) \cup A_\sigma(y_4u_4)| \ge 10$.  We then obtain a good coloring of $G$ by SDR.

Thus, $y_2y_4 \in E(G)$ so that $|A_\sigma(y_2u_2)|, |A_\sigma(y_4u_4)| \ge 8$.  If $y_2x_4 \in E(G)$, then $|A_\sigma(y_2u_2)| \ge 11$, and we obtain a good coloring of $G$ by SDR.  Thus, $y_2x_4 \notin E(G)$, and by symmetry, $x_2y_4 \notin E(G)$.  By the previous case, we have $|A_\sigma(y_2u_2) \cup A_\sigma(x_4u_4)|, |A_\sigma(x_2u_2) \cup A_\sigma(y_4u_4)| \ge 13$, and we obtain a good coloring of $G$ by SDR.

\begin{case}
We cannot extend $\phi$ by coloring an edge of $\sX_1$  and an edge of $\sX_3$ with the same color.
\end{case}

By symmetry, we may assume that the same holds for edges in $\sX_2$ and $\sX_4$.  By Lemma \ref{lem:NoK_2,3}, we may assume that $x_1x_3, x_2x_4 \notin E(G)$.  Thus, by the previous case, $|A_\phi(x_1u_1) \cup A_\phi(x_3u_3)|, |A_\phi(x_2u_2) \cup A_\phi(x_4u_4)| \ge 12$.

Now, if we have at least ten colors available over all the edges in $\sY$, then we can obtain a good coloring of $G$ by SDR.  So, we have $A_\phi(e_1) = A_\phi(e_2)$ and $|A_\phi(e_1)| = 9$ for $e_1,e_2 \in \sY$.  Thus, as above, we may assume that $\Upsilon_\sigma(x_i,u_i) \cup \Upsilon_\sigma(y_i,u_i) =  \{1,2,3,4,5,6\}$ for $i = 1,3$, and $\Upsilon_\sigma(x_i,u_i) \cup \Upsilon_\sigma(y_i,u_i) = \{7,8,9,10,11,12\}$ for $i = 2,4$.

If $y_1y_3 \in E(G)$, then $|A_\phi(y_1u_1)|, |A_\phi(y_3u_3)| \ge 9$, and we obtain a good coloring of $G$ by SDR.  So, $y_1y_3 \notin E(G)$.  By the previous case, we have $|A_\phi(y_1u_1) \cup A_\phi(y_3u_3)| \ge 12$, and we obtain a good coloring of $G$ by SDR.

Thus, in any case, we can extend $\phi$ to a good coloring of $G$.
\end{proof}


\begin{proof}[Proof of Lemma~\ref{le2}]
By the previous lemmas, we know that the girth of $G$ is at least five.  So suppose on the contrary that $u_1u_2u_3u_4u_5$ is a 5-cycle.  By Lemmas \ref{le1}, \ref{lem:NoTriangle}, and \ref{lem:No4cycle}, each $u_i$ has neighbors $x_i,y_i$ not on this 5-cycle. Furthermore, $x_1, y_1, \dots, x_5,y_5$ are distinct and the only possibly adjacencies are between $\{x_i,y_i\}$ and $\{x_{i\pm 2}, y_{i\pm 2}\}$, $i \in [5]$ modulo 5.  However, by Lemma \ref{lem:No4cycle}, neither $x_i$ nor $y_i$ can be adjacent to both $x_{i+2}$ and $y_{i+2}$ (similarly, $x_{i-2}$ and $y_{i-2}$).  As a result, we may assume that $x_2y_4, x_4y_1, x_1y_3, x_3y_5 \notin E(G)$.

Let $G'$ be the graph obtained from $G$ by removing $u_1,u_2,u_3,u_4,u_5$.  By the minimality of $G$, $G'$ has a good coloring $\phi$.  Let $\sX$ denote $\{x_iu_i, y_iu_i: i \in [5]\}$ and $\sY$ denote $\{u_iu_{i+1}: i \in [5] \textrm{ modulo } 5\}$.  Observe that $|A_\phi(e)| \ge 9$ for $e \in \sY$ and $|A_\phi(e')| \ge 6$ for $e' \in \sX$.

Since $x_1y_3 \notin E(G)$, if $A_\phi(x_1u_1) \cap A_\phi(y_3u_3) \neq \emptyset$, then we can color edges $x_1u_1, y_3u_3$ with the same color.  Similarly for the other three nonadjacencies.  Let $\sS := \{\{x_2u_2, y_4u_4\}$, $\{x_4u_4, y_1u_1\}$,$\{x_1u_1, y_3u_3\}$, $\{x_3u_3,y_5u_5\}\}$ so that each element of $\sS$ is a pair of edges that can possibly receive the same color.  Let $\sS' \subseteq \sS$ such that we can extend $\phi$ by coloring each pair of edges in $\sS'$ with its own color, and suppose that $\sS'$ is as large as possible.  Color each pair of edges in $\sS'$ with its own color, and call this good partial coloring $\sigma$.

\begin{case}
$\sS' = \sS$.
\end{case}

Observe that $|A_\sigma(e)| \ge 5$ for $e \in \sY$, and $|A_\sigma(y_2u_2)|, |A_\sigma(x_5u_5)| \ge 2$.  Suppose there exists $\alpha \in A_\sigma(y_2u_2) \cap A_\sigma(u_4u_5)$.  We then color $y_2u_2, u_4u_5$ with $\alpha$, and then  $x_5u_5, u_5u_1, u_1u_2, u_3u_4, u_2u_3$ in this order to obtain a good coloring of $G$.  Thus, $|A_\sigma(y_2u_2) \cup A_\sigma(u_4u_5)| \ge 7$, and by symmetry, $|A_\sigma(x_5u_5) \cup A_\sigma(u_2u_3)| \ge 7$.  We then obtain a good coloring of $G$ by SDR.

\begin{case}
$\sS' \subset \sS$.
\end{case}

Let $k' := |\sS'|$ so that $0 \le k' \le 3$.  Observe that $|A_\sigma(e)| \ge 9 - k'$ for $e \in \sY$, and $|A_\sigma(e')| \ge 6 - k'$ for uncolored $e' \in \sX$.  Since $\sS'$ is the largest subset of $\sS$ that we can color, we have $|A_\sigma(g) \cup A_\sigma(g')| \ge 12 - 2k'$ for all $\{g,g'\} \in \sS\setminus\sS'$.

Since $k' \le 3$, there exists some uncolored $f \in \sX\setminus\{y_2u_2, x_5u_5\}$ and $h \in \sY\setminus\{u_2u_3, u_4u_5\}$ such that $f$ and $h$ can receive the same color if $A_\sigma(f) \cap A_\sigma(h) \neq \emptyset$.  Since $f \notin \{y_2u_2, x_5u_5\}$ and $k' \le 3$, there exists an uncolored edge $f' \in \sX$ such that $\{f,f'\} \in \sS$.

Let $\sT := \{\{y_2u_2, u_4u_5\}, \{x_5u_5,u_2u_3\}, \{f,h\}\}$ so that every element of $\sT$ is a pair of edges that can possibly receive the same color.  Let $\sT' \subseteq \sT$ such that we can extend $\sigma$ by coloring each pair of edges in $\sT'$ with its own color, and suppose that $\sT'$ is as large as possible.  Color each pair of edges in $\sT'$ with its own color, and call this good partial coloring $\psi$.   Let $t' := |\sT'|$ so that $0 \le t' \le 3$.

Let $\sX' \subset \sX$ and $\sY' \subset \sY$ be the edges colored by $\psi$.  So, $|\sX\setminus\sX'| = 10 - 2k' - t'$ and $|\sY\setminus\sY'| = 5 - t'$.  Additionally, $|A_\psi(e_x)| \ge 6 - k' - t'$ for all $e_x \in \sX\setminus\sX'$ and $|A_\psi(e_y)| \ge 9 - k' - t'$ for all  $e_y \in \sY\setminus\sY'$.

We now show that we can obtain a good coloring of $G$ by SDR.  Let $\sA$ be a nonempty subset of $(\sX\setminus\sX') \cup (\sY\setminus\sY')$, and let $\bigcup_\sA := \bigcup_{e \in \sA} A_\psi(e)$.  So $1 \le |\sA| \le 15 - 2k' - 2t'$, and we aim to show that $|\bigcup_\sA| \ge |\sA|$.

\begin{subcase}
$\psi$ does not color $f$ or $h$.
\end{subcase}

Observe that $t' \le 2$.  If $1 \le |\sA| \le 6 - k' - t'$, then $|\bigcup_\sA| \ge 6 - k' - t'$.

If $7' - k' - t' \le |\sA| \le 9 - k' - t'$ and $\sA \cap (\sY\setminus\sY') \neq \emptyset$, then $|\bigcup_\sA| \ge 9 - k' - t'$.  So, we may assume $\sA \subseteq (\sX\setminus\sX')$.  Since $|\sS\setminus\sS'| = 4 - k' \ge 1$, if $\sA$ contains at least $7 - k' - t'$ edges from $\sX\setminus\sX'$, it must include a pair of edges, say $e_s, e_s'$, that form an element of $\sS\setminus\sS'$.  Thus, $|\bigcup_\sA| \ge |A_\psi(e_s) \cup A_\psi(e_s')| \ge 12 - 2k' - t'$, otherwise we could have colored $e_s$ and $e_s'$ with the same color and obtained a larger $\sS' \subseteq \sS$.

Since $|\sT\setminus\sT'| = 3 - t' \ge 1$, if $\sA$ contains at least $13 - 2k' - t'$ edges, it must include a pair of edges, say $e_t, e_t'$, that form an element of $\sT\setminus\sT'$.  Thus, $|\bigcup_\sA| \ge |A_\psi(e_t) \cup A_\psi(e_t')| \ge 15 - 2k' - 2t'$, otherwise we could have colored $e_t$ and $e_t'$ with the same color and obtained a larger $\sT\subseteq\sT'$.

So, it remains to consider when $10 - k' - t' \le |\sA| \le 12 - 2k' - t'$.  Thus, if $k' = 3$, we are done and obtain a good coloring of $G$ by SDR.  So $k' \le 2$.   Observe that $|(\sS\setminus\sS') \cup (\sT\setminus\sT')| = 7 - k' - t'$, and the only edge that is contained in an element of both $\sS\setminus\sS'$ and $\sT\setminus\sT'$ is $f$ ($\{f,f'\} \in \sS\setminus\sS'$ and $\{f,h\} \in \sT\setminus\sT'$).  Thus, we can find $6 - k' - t'$ elements in $(\sS\setminus\sS') \cup (\sT \setminus\sT')$ that are pairwise disjoint.

As a result, when $|\sA| \ge 10 - k' - t'$, $\sA$ must contain a pair of edges, say $e, e'$, that forms an element of $(\sS\setminus\sS') \cup (\sT\setminus\sT')$.  Thus, $|\bigcup_\sA| \ge |A_\psi(e) \cup A_\psi(e')| \ge 12 - 2k' - t' \ge 10 - k' - t'$ for $k \le 2$.  So, in any case, we obtain a good coloring of $G$ by SDR.

\begin{subcase}
$\psi$ colors both $f$ and $h$.
\end{subcase}

Observe that $t' \ge 1$ and $f' \in \sX\setminus\sX'$.  If $1 \le |\sA| \le 6 - k' - t'$, then $|\bigcup_\sA| \ge 6 - k' - t'$.

If $|\sA| = 7 - k' - t'$ and either $f' \in \sA$ or $\sA \cap (\sY\cap\sY') \neq \emptyset$, then $|\bigcup_\sA| \ge 7 - k' - t'$.  So, we may assume $\sA \subseteq \sX\setminus(\sX' \cup \{f'\})$.  Since there are exactly $3 - k'$ uncolored pairs in $\sS\setminus\sS'$, if $\sA$ contains at least $7 - k' - t'$ edges from $\sX\setminus(\sX'\cup\{f'\})$, it must include a pair of edges, say $e_s, e_s'$, that form an element of $\sS\setminus\sS'$.  Thus, $|\bigcup_\sA| \ge |A_\psi(e_s) \cup A_\psi(e_s')| \ge 12 - 2k' - t' \ge 7' - k' - t'$.

If $8 - k' - t' \le |\sA| \le 9 - k' - t'$ and $\sA \cap (\sY\setminus\sY') \neq \emptyset$, then $|\bigcup_\sA| \ge 9 - k' - t'$.  So we may assume $\sA \subseteq (\sX\setminus\sX')$.  However, in a similar manner to the above, $\sA$ must contain a pair of edges that form an  element of $\sS\setminus\sS'$.  Thus, $|\bigcup_\sA| \ge 12 - 2k' - t' \ge 9 - k' - t'$.

So, it remains to consider when $10 - k' - t' \le |\sA| \le 15 - 2k' - 2t'$.  Suppose that $t' \le 2$ so that $|\sT\setminus\sT'| = 3 - t' \ge 1$.  As in the previous subcase, if $\sA$ contains at least $13 - 2k' - t'$ edges, it contains a pair of edges, say $e_t, e_t'$, that form an element of $\sT\setminus\sT'$.  Thus, $|\bigcup_\sA| \ge |A_\psi(e_t) \cup A_\psi(e_t')| \ge 15 - 2k' - 2t'$.  So, $10 - k' - t' \le |\sA| \le 12 - 2k' - t'$.  If $k' = 3$, we are done and obtain a good coloring of $G$ by SDR.  If $k' \le 2$, then as in the previous subcase, we can find $6 - k' - t'$ elements in $(\sS\setminus\sS') \cup (\sT\setminus\sT')$ that are pairwise disjoint.  Thus, when $|\sA| \ge 10 - k' - t'$, $\sA$ must contain a pair of edges that form an element of $(\sS\setminus\sS') \cup (\sT\setminus\sT')$, and $|\bigcup_\sA| \ge 12 - 2k' - t' \ge 10 - k' - t'$ for $k \le 2$.  So, when $t' \le 2$, we obtain a good coloring of $G$ by SDR.

When $t' = 3$, we consider $7 - k' \le |\sA| \le 9 - 2k'$.  If $k' = 3$, we are done and obtain a good coloring of $G$ by SDR.  When $k' \le 2$, $|\sS\setminus\sS'| = 3 - k' \ge 1$ so that if $\sA$ contains at least $7 - k'$ edges, it contains a pair of edges, say $e_s, e_s'$ that form an element of $\sS\setminus\sS'$, and $|\bigcup_\sA| \ge 9 - 2k'$.

Thus, in any case we obtain a good coloring of $G$ by SDR.
\end{proof}

\section{Closing remarks}
The essential part of the proof is to get a nice partition of the vertices described in the introduction.  This partition is largely due to some kind of non-trivial edge-cuts. The study of existence of such edge-cuts may be of independent interest. 

\section*{Acknowledgement}
 
 The authors are very thankful to the referees for their valuable comments.

%
%

\end{document}